\documentclass[twoside,12pt]{article} 
\usepackage[margin=1.0in]{geometry}

\usepackage{amssymb,amsmath,amsthm}
\usepackage{mathtools}
\usepackage{bm,lscape}
\usepackage{bbm}
\usepackage{mathrsfs,dsfont}
\usepackage{wasysym}
\RequirePackage[
    colorlinks=true,
    linkcolor=blue,
    urlcolor=blue,
    citecolor=blue]{hyperref}
\RequirePackage{natbib}
\usepackage[utf8]{inputenc}
\usepackage[english]{babel}
\usepackage{algorithm}
\usepackage{algpseudocode}

\usepackage{graphicx}
\usepackage{color}
\usepackage{rotating}
\usepackage{authblk}
\usepackage{appendix}

\allowdisplaybreaks

\newcommand{\dR}{\mathbb{R}}
\newcommand{\dC}{\mathbb{C}}

\newcommand{\dP}{\mathbb{P}}
\newcommand{\cB}{\mathcal{B}}
\newcommand{\cF}{\mathcal{F}}
\newcommand{\cL}{\mathcal{L}}

\newcommand{\dE}{\mathbb{E}}

\allowdisplaybreaks

\newcommand{\E}{\mathds{E}}

\newcommand{\ddr}{\mathrm{d}}

\newtheorem{thm}{Theorem}[section]
\newtheorem{lem}[thm]{Lemma}

\providecommand{\keywords}[1]
{
  \small	
  \textbf{\textit{Keywords:}} #1
}

\begin{document}

\title{A new look on large deviations and concentration inequalities for the Ewens-Pitman model}


\author[1]{Bernard Bercu\thanks{bernard.bercu@math.u-bordeaux.fr}}
\author[2]{Stefano Favaro\thanks{stefano.favaro@unito.it}}
\affil[1]{\small{Institut de Math\'ematiques de Bordeaux, Universit\'e de Bordeaux, France}}
\affil[2]{\small{Department of Economics and Statistics, University of Torino and Collegio Carlo Alberto, Italy}}

\maketitle

\begin{abstract}
The Ewens-Pitman model is a probability distribution for random partitions of the set $[n]=\{1,\ldots,n\}$, parameterized by $\alpha\in[0,1)$ and $\theta>-\alpha$, with $\alpha=0$ corresponding to the Ewens model in population genetics. The goal of this paper is to provide an alternative and concise proof of the Feng-Hoppe large deviation principle for the number $K_{n}$ of partition sets in the Ewens-Pitman model with $\alpha\in(0,1)$ and $\theta>-\alpha$. Our approach leverages an integral representation of the moment-generating function of $K_{n}$ in terms of the (one-parameter) Mittag-Leffler function, along with a sharp asymptotic expansion of it. This approach significantly simplifies the original proof of Feng-Hoppe large deviation principle, as it avoids all the technical difficulties arising from a continuity argument with respect to rational and non-rational values of $\alpha$. Beyond large deviations for $K_{n}$, our approach allows to establish a sharp concentration inequality for $K_n$ involving the rate function of the large deviation principle.
\end{abstract}



\section{Introduction}

The Ewens-Pitman model was introduced in \cite{Pit(95)} as a generalization of the Ewens model in population genetics \cite{Ewe(72)}. Let $n\in\mathbb{N}$, and consider a random partition of $[n]=\{1,\ldots,n\}$ into $K_{n}\in\{1,\ldots,n\}$ partition sets of which 
$K_{r,n}\in\{0,1,\ldots,n\}$ appear $r$ times, 
$$
n=\sum_{r=1}^{n}rK_{r,n}\qquad \text{and}\qquad K_{n}=\sum_{r=1}^{n}K_{r,n}.
$$ 
For $\alpha\in[0,1)$ and $\theta>-\alpha$, the Ewens-Pitman model assigns to $\mathbf{K}_{n}=(K_{1,n},\ldots,K_{n,n})$ the probability 
\begin{equation}\label{epsm}
\mathbb{P}_{\alpha,\theta}(\mathbf{K}_{n}=(k_{1},\ldots,k_{n}))=n!\frac{\left(\frac{\theta}{\alpha}\right)^{(\sum_{i=1}^{n}k_{i})}}{(\theta)^{(n)}}\prod_{i=1}^{n}\left(\frac{\alpha(1-\alpha)^{(i-1)}}{i!}\right)^{k_{i}}\frac{1}{k_{i}!},
\end{equation}
where, for any $a\in \dR$, $(a)^{(n)}$ stands for the rising factorial of $a$ of order $n$, that is 
$(a)^{(n)}=a(a+1)\cdots(a+n-1)$. The Ewens model is recovered from \eqref{epsm} by setting $\alpha=0$. The distribution \eqref{epsm} admits a sequential construction through the Chinese restaurant process \cite{Pit(95),Fen(98)}, a Poisson process construction by random sampling from the Pitman-Yor random measure \cite{Per(92),Pit(97)}, and a construction through compound Poisson random partitions \cite{Dol(21)}. We refer to \cite[Chapter 3 and Chapter 4]{Pit(06)} for an overview of the Ewens-Pitman model, including applications in a variety of research fields, e.g., population genetics, excursion theory, Bayesian statistics, combinatorics and statistical physics.

\subsection{Background}

There exists a rich body of literature on the large $n$ asymptotic behaviour of the number $K_{n}$ of partition sets in the Ewens-Pitman model, e.g., almost-sure and Gaussian fluctuations, large and moderate deviations, Berry-Esseen inequalities, laws of iterated logarithm, central limit theorems and laws of large numbers \cite{Kor(73),Pit(95),Pit(97),Fen(98),Pit(06),Fav(15),Fav(18),Dol(20),Dol(21),Ber(24),Con(24)}. Feng-Hoppe large deviation principle is one of the earliest results in the literature \cite{Fen(98)}. For $\alpha=0$, $K_{n}$ is a sum of independent Bernoulli random variables \cite{Kor(73)}, from which it follows easily that ($K_n/n)$ satisfies a large deviation principle with good rate function
\begin{displaymath}
\label{RATEF0}
   I_\theta(x) =\left \{ \begin{array}{ll}
    {\displaystyle x\log\Big(\frac{x}{\theta}\Big)-x+\theta}  & \text{ if } \quad x>0 \vspace{1ex}\\[0.2cm]
     \theta & \text{ if } \quad x=0 \vspace{1ex}\\[0.2cm]
     +\infty & \text{ otherwise.} 
   \end{array}  \right.
\vspace{1ex}
\end{displaymath}
For $\alpha\in(0,1)$, $K_{n}$ is no longer a sum of independent Bernoulli random variables. However, by a careful analysis of the moment-generating function of $K_{n}$, \cite{Fen(98)} proved the following theorem.

\begin{thm}[Feng-Hoppe large deviation principle]\label{T-LDP}
Under the Ewens-Pitman model with $\alpha\in(0,1)$ and $\theta>-\alpha$, the sequence ($K_n/n)$ satisfies a large deviation principle with good rate function
\begin{equation}
\label{RATEF}
   I_\alpha(x) =\sup_{ t \in \dR} \left\{xt - L_\alpha(t) \right\},
\end{equation}
where
\begin{displaymath}
\label{DEFL}
L_\alpha(t) =\left \{ \begin{array}{ll}
 -\log \bigl(1-(1-e^{-t})^{1/\alpha}\bigr) & \text{ if } \quad t>0 \vspace{1ex}\\[0.2cm]
     0 & \text{ otherwise. }  
     \end{array}  \right.
\end{displaymath}
\end{thm}

The proof of Theorem \ref{T-LDP} is lengthy and technically involved. Its core consists in showing that 
\begin{equation}\label{loglap}
\vspace{1ex}
\lim_{n\rightarrow+\infty}\frac{1}{n}\log\E_{\alpha,\theta}[\exp\{t K_{n}\}]=L_{\alpha}(t),
\end{equation}
from which Theorem \ref{T-LDP} follows using G\"artner--Ellis theorem \cite[Theorem 2.3.6]{Dem(98)}. Assuming $\theta=0$, the proof of \eqref{loglap} starts from the calculation of the moment-generating function $m_{n}(t)$ of $K_{n}$, and proceeds along three steps: i) for $\alpha$ rational, the large $n$ behaviour of $m_{n}(t)$ is derived through tedious algebra and combinatorics; ii) for non-rational $\alpha$, upper and lower large $n$ approximations of $m_{n}(t)$ are constructed by leveraging monotonicity and considering nearby rational $\alpha$; iii) a continuity argument is applied to bridge the gap between rational and non-rational values of $\alpha$. The proof of \eqref{loglap} is completed by reducing the case $\theta\neq0$ to $\theta=0$, though the argument could benefit from a greater clarity.

\subsection{Our contribution}

In this paper, we present a new approach to large deviations for $K_{n}$, offering an alternative and concise proof of Theorem \ref{T-LDP}. Assuming $\theta=0$, the novelty of our approach lies in the use of an integral representation of $m_{n}(t)$ in terms of the (one-parameter) Mittag-Leffler function \cite[Chapter 3]{Gor(14)}. Then, by leveraging a well-known integral representation of the Mittag-Leffler function \cite{Gor(02)}, we obtain a sharp expansion of $m_{n}(t)$, which leads to \eqref{loglap} through straightforward algebra. In doing this, we also clarify and make rigorous how the case $\theta\neq0$ reduces to the case $\theta=0$. This approach significantly simplifies the original proof Theorem \ref{T-LDP}, as it allows to avoid the technical difficulties arising from the continuity argument with respect to rational and non-rational values of $\alpha$. 

Beyond large deviations for $K_{n}$, our approach allows to establish a sharp concentration inequality for $K_n$ involving the rate function $I_\alpha$ in \eqref{RATEF} as stated in the next theorem.

\begin{thm}\label{T-CI}
Under the Ewens-Pitman model with $\alpha\in(0,1)$ and $\theta>-\alpha$, for any $x \in(0,1)$ and all $n \geq 1$,
\begin{equation}
\label{CI}
   \dP_{\alpha,\theta} \left( \frac{K_n}{n} \geq x \right) \leq \frac{ P_n( \alpha, \theta) c_n( \alpha, \theta)}{\alpha} \exp\left(\! -n I_\alpha(x)\right),
\end{equation}
where
\begin{displaymath}
P_n( \alpha, \theta)=\left \{ \begin{array}{ll}
    {\displaystyle \big(\lfloor \theta_\alpha \rfloor + n \big)^{\theta_\alpha}}  & \text{ if } \quad \theta>0 \vspace{1ex}\\[0.2cm]
    1 & \text{ if } \quad \theta \leq 0
   \end{array}  \right.
\end{displaymath}
and
\begin{equation}
\label{DEFCN}
c_n( \alpha, \theta)= 
\frac{\Gamma(n)\Gamma(\theta +1)}{\Gamma\left(\theta_\alpha+1\right)\Gamma(n+\theta)},
\end{equation}
with $\theta_{\alpha}=\theta/\alpha$ and $\lfloor \theta_\alpha \rfloor$ the integer part of $\theta_\alpha$. In particular, if $\theta=0$ then $P_n( \alpha, \theta) c_n( \alpha, \theta)=1$.
\end{thm}

\subsection{Organization of the paper}

The paper is structured as follows. In Section \ref{sec1} we present an alternative derivation of the moment-generating function of $K_{n}$, through the sequential construction of the Ewens-Pitman model, as well as a sharp expansion of it. Such an expansion is the key ingredient to obtain an alternative and concise proof of Theorem \ref{T-LDP} in Section \ref{sec2}, and to prove Theorem \ref{T-CI} in Section  \ref{sec3}. In Section \ref{sec4} we conclude by discussing some directions for future work.


\section{On the moment-generating function of $K_{n}$}\label{sec1}

For $\alpha\in(0,1)$ and $\theta>-\alpha$, let $\mathbb{P}_{\alpha,\theta}$ be the law of the Pitman-Yor random measure, which induces the Ewens-Pitman model \eqref{epsm}, and let $\mathbb{E}_{\alpha,\theta}$ be the corresponding expectation. For an integer $n\geqslant 1$, denote by $m_n$ the moment-generating function of $K_n$ defined, for all $t \in \dR$, by
\begin{equation*}
m_n(t)=\dE_{\alpha,\theta}[\exp(t K_n)].
\end{equation*}
From a change of measure formula that links $\mathbb{P}_{\alpha,\theta}$, with $\theta>-\alpha$, and $\mathbb{P}_{\alpha,0}$ \cite[Section 3]{Pit(97)}, denote
\begin{equation} \label{MARTPITMAN}
M_n(\alpha, \theta)= 
\frac{a_n \Gamma\left(\theta_\alpha+K_n\right)}{\Gamma\left(\theta_\alpha+1\right) \Gamma(K_n)},
\end{equation}
where $\theta_{\alpha}=\theta/\alpha$ and
\begin{displaymath}
a_n=\prod_{k=1}^{n-1} \left(\frac{k}{k+\theta}\right)=\frac{\Gamma(n)\Gamma(\theta +1)}{\Gamma(n+\theta)}.
\end{displaymath}
It is not difficult to see that the sequence $(M_n(\alpha, \theta))$ is a positive $\mathbb{P}_{\alpha,0}$-martingale with mean value $\dE_{\alpha,0}[M_n(\alpha, \theta)]=1$. Consequently, we deduce directly from \eqref{MARTPITMAN} that for all $t \in \dR$,
\begin{equation}\label{mgf_general}
m_n(t)=\dE_{\alpha,\theta}[\exp(t K_n)]=\dE_{\alpha,0}[\exp(t K_n)M_n(\alpha, \theta)]=c_n(\alpha,\theta) I_{n}(t),
\end{equation}
where $c_n(\alpha,\theta)$ is given in \eqref{DEFCN} and
\begin{displaymath}
I_n(t)=\dE_{\alpha,0}\left[\exp(t K_n)\frac{\Gamma\left(\theta_\alpha+K_n\right)}{\Gamma(K_n)}\right].
\end{displaymath}
One can observe that the moment-generating $m_n(t)$ is finite for all $t \in \dR$ and for all $n \geq 1$, since $K_n\leq n$. 

\subsection{The case $\theta=0$}

For $\alpha\in(0,1)$ and $\theta=0$, \cite[Equation (3.5)]{Fen(98)} provides an expression for $m_{n}(t)$ in terms of the rising factorial moments of $K_{n}$ \cite[Chapter 3]{Pit(06)}. In particular for all $n\geq1$ and all $t>0$,
\begin{equation}\label{SHARPMNT1}
m_{n}(t)= \frac{1}{(n-1)!} \sum_{\ell=0}^\infty \left(\frac{e^t -1}{e^t}\right)^\ell \frac{\Gamma(n+\alpha \ell)}{\Gamma( \alpha \ell+1)}.
\end{equation}
We propose an alternative proof of \eqref{SHARPMNT1}, which is of independent interest. Let $F$ be the generating function of the sequence $(m_{n}(t))$, which is given, for all $t \in \dR$
and for all $z \in \dC$, by
\begin{equation}\label{DEFGF}
F(t,z)=\sum_{n=1}^\infty m_n(t) z^{n-1}.
\end{equation}
The radius of convergence of $F$, denoted $R^F(t)$, should depend on $t$ and it is positive since, for all $t \in \dR$, $|m_n(t)| \leq e^{n |t|}$. The next lemma provides an expression for $F$ by relying on the definition of $K_{n}$ in terms of the sequential construction of the Ewens-Pitman model \cite{Pit(95),Fen(98)}.

\begin{lem}\label{L-GF} 
For $\alpha\in(0,1)$ and $\theta=0$, we have for all $t\in \dR$,
\begin{equation}\label{RF}
R^F(t) =  1+ \left|\frac{e^t-1}{e^t}\right|^{1/\alpha}.
\end{equation}
Moreover, for all $t \in \dR$ and for all $z \in \dC$ such that $|z| < R^F(t)$,
\begin{equation}\label{SOLF}
F(t,z)=\frac{1}{1-z} \left(1- \frac{(e^t -1)}{(e^t -1)-e^t(1-z)^{\alpha}}\right).
\end{equation}
\end{lem}

\begin{proof}
It follows from \cite[Proposition 8]{Pit(95)}, for all $n \geq 1$, $K_{n+1} = K_n + \xi_{n+1}$ where 
the conditional distribution $\cL(\xi_{n+1}|\cF_n)$ is the Bernoulli $\cB(p_n)$ distribution with $p_{n}=\alpha K_{n}/n$. Consequently, for all $t \in \dR$ and for all $n \geq 1$, 
\begin{align*}
m_{n+1} (t) = \dE_{\alpha,0}\left[\exp(t K_{n+1})\right]&=\dE_{\alpha,0}\left[\exp(t K_n)\dE_{\alpha,0}[\exp(t\xi_{n+1})|\cF_n]\right]\\
& = \dE_{\alpha,0}\left[\exp(t K_n) p_n e^t+\exp(t K_n)(1-p_n)\right]\\
& =   m_n(t)+ (e^t -1) \dE_{\alpha,0}\left[p_n\exp(t K_n)\right].
\end{align*}
As $p_n= \alpha K_n/n$, we have $\dE_{\alpha,0}\left[p_n\exp(t K_n)\right]=\alpha m_n^\prime(t)/n$, 
which implies that for all $t \in \dR$ and for all $n\geq 1$,
\begin{equation}
\label{LAPLACE1}
m_{n+1}(t)=m_n(t)+\frac{\alpha}{n} (e^t -1) m_n^\prime(t).
\end{equation}
From \eqref{DEFGF} and \eqref{LAPLACE1}, we obtain for all $|z|<R^F(t)$
\begin{align*}
\frac{\partial F(t,z)}{\partial z} = \sum_{n=2}^\infty (n-1) m_n(t) z^{n-2}&=\sum_{n=1}^\infty 
n m_{n+1}(t) z^{n-1} \\
& =  \sum_{n=1}^\infty nm_n(t) z^{n-1} + \sum_{n=1}^\infty \alpha(e^t-1)m_n^\prime(t)z^{n-1} \\
& =   F(t,z)+ z\frac{\partial F(t,z)}{\partial z} + \alpha(e^t-1)\frac{\partial F(t,z)}{\partial t},  
\end{align*}
which leads to
\begin{equation}
\label{DIFEQF}
(1- z)\frac{\partial F(t,z)}{\partial z} + \alpha(1-e^t) \frac{\partial F(t,z)}{\partial t} 
=  F(t,z)
\end{equation}
with two initial values
\begin{equation}
\label{INITVAL}
F(t,0)=e^t \hspace{1cm}\text{and}\hspace{1cm}F(0,z)=\frac{1}{1-z}.
\end{equation}
Now, the classical method of characteristics allows to solve the first-order partial differential equation \eqref{DIFEQF}. We can associate \eqref{DIFEQF} to
the ordinary differential system given by
\begin{displaymath}
    \frac{dz}{1-  z} = \frac{dt}{\alpha(1-e^t)} = \frac{dw}{w},
\end{displaymath}
where $w$ stands for $F$. We assume in the sequel that $t>0$, inasmuch as the proof for $t<0$ follows exactly the same lines. On the one hand, the equation binding $w$ and $t$ is such that
\begin{equation}\label{ODIFSYST2}
    \frac{dw}{w} = \frac{dt}{\alpha(1-e^t)}.
\end{equation}
From \eqref{ODIFSYST2}, the solution is
\begin{equation} \label{EQC1}
w = C_1 \left(\frac{e^t}{e^{t}-1}\right)^{1/\alpha}.
\end{equation}
On the other hand, the equation binding $z$ and $t$ is similar but inhomogeneous, that is we write
\begin{equation}\label{ODIFSYST3}
\frac{dz}{dt} = \frac{1}{\alpha(e^t-1)} z - \frac{1}{\alpha(e^t-1)}.
   \end{equation}
We obtain from \eqref{ODIFSYST3}, by means of  the variation of constant method, that the solution satisfies
\begin{equation}
\label{EQC2}
z =1+C_2\left(\frac{e^t-1}{e^{t}}\right)^{1/\alpha}.
\end{equation}
By coupling \eqref{EQC1} and \eqref{EQC2},
\begin{equation}\label{EQPSI}
 C_1 = f(-C_2),
\end{equation}
where $f$ remains to be explicitly calculated using the boundary values \eqref{INITVAL}.
We get from \eqref{EQC1} that
\begin{displaymath}
C_1=F(t,z)\left(\frac{e^t-1}{e^{t}}\right)^{1/\alpha}.
\end{displaymath}
Hence, we deduce 
from \eqref{EQC1}, \eqref{EQC2} and \eqref{EQPSI} that
for all $t>0$ and for all $z \in \dC$ such that $|z|<R^F(t)$,
\begin{equation}
\label{EQFPSI}
F(t,z)=\left(\frac{e^t}{e^{t}-1}\right)^{1/\alpha}f\left(\left(\frac{e^t}{e^{t}-1}\right)^{1/\alpha}(1-z)\right)=c(t)f(c(t)(1-z)),
\end{equation}
where
\begin{displaymath}
c(t)=\left(\frac{e^t}{e^{t}-1}\right)^{1/\alpha}.
\end{displaymath}
We already saw from \eqref{INITVAL} that for $z=0$, $F(t,0)=e^t$, which implies from \eqref{EQFPSI} with the special value $z=0$ that $e^t=c(t)f(c(t))$. Hence, we find from
the change of variable $s=c(t)$ that for all $s>1$,
\begin{equation}
\label{DEFf}
f(s) = \frac{s^\alpha}{s(s^\alpha-1)}.
\end{equation}
Finally, it follows from \eqref{EQFPSI} and \eqref{DEFf} that
for all $t>0$ and for all $z \in \dC$ such that $|z|<R^F(t)$,
\begin{displaymath}
F(t,z)=\frac{1}{1-z}\left(\frac{\left(c(t) (1-z) \right)^\alpha}{\left(c(t) (1-z) \right)^\alpha-1}\right)= \frac{1}{1-z}\left(1- \frac{1}{1-\left(c(t) (1-z) \right)^\alpha}\right),
\end{displaymath}
leading to \eqref{SOLF}. 
The radius of convergence \eqref{RF} comes from \eqref{SOLF}, which completes the proof.
\end{proof}

From now on, the calculation of \eqref{SHARPMNT1} follows from a direct application of Lemma \ref{L-GF}. In particular, consider the function $\cF$ defined, for all $t >0$ and for all $z \in \dC$ such that $|z| < R^F(t)$, by
\begin{equation}\label{DEFCF}
\cF(t,z)= (1-z) F(t,z)=\varphi( \xi(t,z)),
\end{equation}
where
\begin{displaymath}
\xi(t,z)= c(t)(1-z)
\end{displaymath}
and
\begin{displaymath}
\varphi(z) = \frac{z^\alpha}{z^\alpha-1}=\frac{1}{1-z^{-\alpha}}.
\end{displaymath}
By choosing $z\in \dC$ such that $|z|<1<|\xi(t,z)|$, we obtain from Newton's generalized Binomial theorem
\begin{align*}
\cF(t,z) =\sum_{k=0}^\infty \left(\xi(t,z) \right)^{-\alpha k} &=\sum_{k=0}^\infty \left(\frac{e^t -1}{e^t}\right)^k(1-z)^{-\alpha k} \\
&=\sum_{k=0}^\infty \left(\frac{e^t -1}{e^t}\right)^k \sum_{n=0}^\infty \frac{(\alpha k)^{(n)}}{n!} z^n \\
&= \sum_{n=0}^\infty \frac{1}{n!} \sum_{k=0}^\infty \left(\frac{e^t -1}{e^t}\right)^k (\alpha k)^{(n)} z^n\\
&= \sum_{n=0}^\infty m_n^{\cF}(t) z^n,
\end{align*}
where for all $n \geq 0$,
\begin{equation}
\label{MNCALFT}
m_n^{\cF}(t)= \frac{1}{n!} \sum_{k=0}^\infty \left(\frac{e^t -1}{e^t}\right)^k (\alpha k)^{(n)}.
\end{equation}
Moreover, we have from \eqref{DEFGF} together with the application of \eqref{DEFCF} that $m_1(t)= e^t$ and for all $n \geq 1$, $m_{n+1}(t)=m_n(t)+m_n^{\cF}(t)$. Accordingly, we deduce from \eqref{MNCALFT}, for all $n \geq 1$, that
\begin{align*}
m_{n+1}(t) =e^t +\sum_{k=1}^n  m_k^{\cF}(t) &=e^t  +\sum_{k=1}^n \frac{1}{k!} \sum_{\ell=0}^\infty \left(\frac{e^t -1}{e^t}\right)^\ell (\alpha \ell)^{(k)} \\
&=e^t  +\sum_{\ell=0}^\infty \left(\frac{e^t -1}{e^t}\right)^\ell  \sum_{k=1}^n \frac{(\alpha \ell)^{(k)} }{k!}\\
& = e^t  +\sum_{\ell=0}^\infty \left(\frac{e^t -1}{e^t}\right)^\ell\left( \frac{1}{n!}(\alpha \ell +1)^{(n)} -1 \right) \\
&= \frac{1}{n!} \sum_{\ell=0}^\infty \left(\frac{e^t -1}{e^t}\right)^\ell (\alpha \ell +1)^{(n)} \\
&= \frac{1}{n!} \sum_{\ell=0}^\infty \left(\frac{e^t -1}{e^t}\right)^\ell 
\frac{\Gamma(n+\alpha \ell +1)}{\Gamma( \alpha \ell+1)},
\end{align*}
which coincides with the expression \eqref{SHARPMNT1}. As we will discuss below, similar arguments extend to the case $\theta\neq0$, for which no explicit expression of $m_{n}(t)$ is available in the literature.

The next lemma provides a sharp expansion of $m_{n}(t)$ in \eqref{SHARPMNT1}. This is the key ingredient to provide an alternative and concise proof of Theorem \ref{T-LDP} and to prove Theorem \ref{T-CI}.

\begin{lem}\label{L-DEVMNT}
For $\alpha\in(0,1)$ and $\theta=0$, we have for all $t>0$ and for all $n \geq 2$,
\begin{equation}\label{MNT}
m_n(t)=\frac{1}{\alpha} 
\left( 1 - \left(\frac{e^t -1}{e^t}\right)^{1/\alpha}\right)^{-n} -R_n(t),
\end{equation}
where
\begin{equation}\label{RNT}
R_n(t)=\frac{1}{(n-1)!}\int_0^{+\infty} y^{n-1} e^{-y} \int_0^{+\infty} G_\alpha(x, (1-e^{-t}) y^\alpha) dx dy
\end{equation}
with
\begin{equation}\label{DEFGalpha}
G_\alpha(x,y) = \frac{1}{\pi \alpha} \exp(-x^{1/\alpha}) \left(\frac{y \sin(\pi \alpha)}{x^2-2xy \cos(\pi \alpha)+y^2}\right).
\end{equation}
\end{lem}
\begin{proof}
It follows from \eqref{SHARPMNT1}, by applying the integral representation of $\Gamma(n+\alpha\ell)$, that for all $n\geq1$ and all $t>0$,
\begin{align}
m_{n}(t)&=\frac{1}{(n-1)!} \sum_{\ell=0}^\infty \frac{(1-e^{-t})^\ell}{\Gamma( \alpha \ell+1)}
\int_0^\infty y^{n+\alpha \ell-1} e^{-y} dy \nonumber \\
&=\frac{1}{(n-1)!} \int_0^\infty y^{n-1} e^{-y} E_\alpha\left( (1-e^{-t})y^\alpha\right) dy,
\label{SHARPMNT2}
\end{align}
where $E_\alpha$ stands for the Mittag-Leffler function \cite[Chapter 3]{Gor(14)}, which is defined, for all $z \in \dC$, by 
$$
E_\alpha(z)=\sum_{\ell=0}^\infty\frac{z^\ell}{\Gamma( \alpha \ell+1)}.
$$  
We deduce from the integral representation of the Mittag-Leffler function $E_\alpha$ given by \cite[Theorem 2.3]{Gor(02)} that for all $z \in \dC$ such that $|\arg(z)|< \pi \alpha$,
\begin{equation}
\label{INTREalpha}
E_\alpha(z)=\frac{1}{\alpha} \exp(z^{1/\alpha}) - \int_0^\infty G_\alpha(x,z) dx,
\end{equation}
where the function $G_\alpha$ is defined in \eqref{DEFGalpha}. Consequently,
we obtain from \eqref{SHARPMNT2} and \eqref{INTREalpha} that
\begin{equation}
\label{SHARPMNT3} 
m_{n}(t)= \frac{1}{\alpha (n-1)!} \int_0^\infty y^{n-1} e^{-y} \exp((1-e^{-t})^{1/\alpha}y)dy -R_n(t),
\end{equation}
where the remainder $R_n(t)$ is given by \eqref{RNT}. In particular, the first term in the right-hand side of \eqref{SHARPMNT3} can be easily calculated. As a matter of fact, by a simple change of variables
\begin{align}
\int_0^\infty \!\! y^{n-1} e^{-y} \exp((1-e^{-t})^{1/\alpha}y)dy 
&=\int_0^\infty \!\! y^{n-1}  \exp(-y(1-(1-e^{-t})^{1/\alpha}))dy \nonumber \\
&=\left( 1 - \left(\frac{e^t -1}{e^t}\right)^{1/\alpha}\right)^{-n} \int_0^\infty \!\!x^{n-1} \exp(-x)dx \nonumber\\
&=(n-1)! \left( 1 - \left(\frac{e^t -1}{e^t}\right)^{1/\alpha}\right)^{-n}.
\label{SHARPMNT4} 
\end{align}
Finally, the compact expression \eqref{MNT} follows from \eqref{SHARPMNT3} and \eqref{SHARPMNT4}, which completes the proof.
\end{proof}

\subsection{The case $\theta\neq0$}

For $\alpha\in(0,1)$ and $\theta>-\alpha$, an extension of Lemma \ref{L-GF} is possible by exploiting the change of measure formula \eqref{mgf_general}. In particular, it allows us to prove that for all $n\geq1$ and for all $t>0$,
\begin{equation}\label{SHARPMNT1THETA}
m_{n}(t)= \frac{\Gamma(\theta+1)e^{-t\theta_\alpha}}{\Gamma(n+\theta)}   
 \sum_{\ell=0}^\infty \frac{1}{\ell!} 
\big(\theta_\alpha+1\big)^{(\ell)} 
\Big(\frac{e^t -1}{e^t}\Big)^\ell  
 \frac{\Gamma(n+\theta+ \alpha \ell)}{\Gamma( \theta+\alpha \ell+1)}.
\end{equation}
The expression \eqref{SHARPMNT1THETA} is new, and we recover \eqref{SHARPMNT1} from \eqref{SHARPMNT1THETA} by setting $\theta=\theta_\alpha=0$. From \eqref{SHARPMNT1THETA}, by applying the integral representations of $\Gamma(n+\theta+\alpha\ell)$, we obtain that for all $n\geq1$ and all $t>0$
\begin{align}\label{SHARPMNT1THETA_int}
\notag m_{n}(t)&=\frac{\Gamma(\theta+1)\text{e}^{-t\theta_{\alpha}}}{\Gamma(n+\theta)}\sum_{\ell=0}^{\infty}\frac{(1-e^{-t})^{\ell}}{\ell!\Gamma(\theta+\alpha\ell+1)}(\theta_{\alpha}+1)^{(\ell)}\int_{0}^{+\infty}y^{n+\theta+\alpha\ell-1}e^{-y}\ddr y\\
&=\frac{\Gamma(\theta+1)\text{e}^{-t\theta_{\alpha}}}{\Gamma(n+\theta)}\int_{0}^{+\infty}y^{n+\theta-1}e^{-y}E^{\theta_{\alpha}+1}_{\alpha,\theta+1}((1-e^{-t})y^{\alpha})\ddr y,
\end{align}
where $E^{\gamma}_{\alpha,\beta}$ stands for the three-parameter Mittag-Leffler function \cite[Chapter 5]{Gor(14)} defined, for all $z\in\mathbb{C}$, by
\begin{displaymath}
E^{\gamma}_{\alpha,\beta}(z)=\sum_{\ell=0}^{\infty}\frac{(\gamma)^{(l)}z^{\ell}}{\ell!\Gamma(\alpha\ell+\beta)}.
\end{displaymath}
This is a generalization of the Mittag-Leffler function, which is recovered by setting  $\gamma=\beta=1$. The expression \eqref{SHARPMNT1THETA_int} is new, and we find again \eqref{SHARPMNT2} from \eqref{SHARPMNT1THETA_int} by setting $\theta=\theta_\alpha=0$.
\ \vspace{1ex}\\
To the best of our knowledge, no integral representation of the three-parameter Mittag-Leffler function, analogous of that in \eqref{INTREalpha} for the Mittag-Leffler function, is available in the literature. In particular, the proof of \cite[Theorem 2.3]{Gor(02)}, in combination with \cite[Equation 3]{Mai(15)}, suggests that such an integral representation may not exists for all values of $\theta_{\alpha}$.


\section{A concise proof of Theorem \ref{T-LDP}}\label{sec2}

First of all, we focus our attention on the special case $\theta=0$. In particular, it follows from the elementary fact that $K_n \geq 1$, together with the use of Jensen's inequality, that for all $t\leq 0$, 
\begin{equation}
\label{JENSENKN}
t \dE_{\alpha,0}[K_n] \leq \log \dE_{\alpha,0}[\exp(tK_n)] \leq t.
\end{equation}
However, we know from \cite[Section 3.3.]{Pit(06)} that
$$
\lim_{n\rightarrow+\infty}\frac{1}{n^{\alpha}} \dE_{\alpha,0}[K_{n}]=\frac{1}{\alpha\Gamma(\alpha)}. 
$$
Hence, we immediately obtain from \eqref{JENSENKN} that
for all $t\leq 0$, 
\begin{equation}
\label{CVGNEGLAPLACE}
\lim_{n\rightarrow \infty} \frac{1}{n} \log m_n(t)=0.
\end{equation}
Hereafter, we assume that $t>0$. For all $x, y \in \dR$, we have that $x^2-2xy \cos(\pi \alpha)+y^2 \geq \sin^2(\pi \alpha)y^2$.
Consequently, the function $G_\alpha$ given by \eqref{DEFGalpha} satisfies for all $x>0$ and $y>0$,
\begin{equation*}
\label{CRUDUBG}
G_\alpha(x,y) \leq \frac{\exp(-x^{1/\alpha})}{\pi \alpha \sin(\pi \alpha) y},
\end{equation*}
which implies that for all $y>0$,
\begin{equation}\label{INTUBG}
\int_0^\infty G_\alpha(x,y) dx \leq \frac{1}{\pi \alpha \sin(\pi \alpha) y}\int_0^\infty \exp(-x^{1/\alpha}) dx\leq \frac{\Gamma(\alpha)}{\pi \sin(\pi \alpha) y}.
\end{equation}
Hence, we deduce from the expression of $R_{n}(t)$ in \eqref{RNT} together with \eqref{INTUBG} that 
\begin{equation}
\label{RNTNUB}
0<R_n(t) \leq 
\frac{ \Gamma(\alpha) \Gamma(n-\alpha)}{\Gamma(n)\pi \sin(\pi \alpha) (1-e^{-t})}
\leq \frac{n^{1- \alpha} \Gamma(\alpha)}{(n-\alpha)\pi \sin(\pi \alpha) (1-e^{-t})}
\end{equation}
where  the right-hand side of \eqref{RNTNUB} is
due to Wendel's inequality, $\Gamma(n+1-\alpha)\leq n^{1-\alpha}\Gamma(n)$. 
Hereafter, denote
$$
d(t)=1 - \left(\frac{e^t -1}{e^t}\right)^{1/\alpha}.
$$
It follows from \eqref{MNT} and \eqref{RNTNUB} that
\begin{equation}
\label{ULBMNT}
\frac{1}{\alpha}\big(d(t) \big)^{-n}
\Big(1- \frac{n^{1-\alpha}\Gamma(1+\alpha) \big(d(t) \big)^{n} }{(n-\alpha)\pi \sin(\pi \alpha) (1-e^{-t})} \Big)< m_n(t) < \frac{1}{\alpha} \big(d(t) \big)^{-n}.
\end{equation}
Therefore, as $0<d(t)<1$, we obtain from the upper and the lower bounds in \eqref{ULBMNT} that
for all $t>0$,
\begin{equation}
\label{CVGLAPLACE0}
\lim_{n\rightarrow \infty} \frac{1}{n}\log m_n(t)=-\log d(t)=L_\alpha(t).
\end{equation}
The function $L_\alpha$ is finite and differentiable on $\mathbb{R}$, with $L_\alpha(t)=0$ for all $t \leq 0$. Then, G\"artner-Ellis theorem \cite[Theorem 2.3.6]{Dem(98)} completes the proof of Theorem \ref{T-LDP} in the special case $\theta=0$. 
It only remains to carry out the proof of Theorem \ref{T-LDP} in the case
$\theta \neq 0$. In particular, one can observe that the convergence \eqref{CVGNEGLAPLACE} also holds for all $t \leq 0$ as
soon as $\theta > -\alpha$. On the one hand, assume that $\theta >0$. Then, we obtain once again from Wendel's inequality that
\begin{align*}
\frac{\Gamma\left(\theta_\alpha+K_n\right)}{\Gamma(K_n)} &\leq \frac{\Gamma\left(\lfloor \theta_\alpha \rfloor+K_n\right)}{\Gamma(K_n)} \big(\lfloor \theta_\alpha \rfloor +K_n \big)^{\{\theta_\alpha\}}\\
& \leq \big(\lfloor \theta_\alpha \rfloor +K_n \big)^{\lfloor \theta_\alpha\rfloor} \big(\lfloor \theta_\alpha \rfloor +K_n \big)^{\{\theta_\alpha\}} \\
&\leq \big(\lfloor \theta_\alpha \rfloor +K_n \big)^{\theta_\alpha}\\
& \leq \big(\lfloor \theta_\alpha \rfloor + n \big)^{\theta_\alpha},
\end{align*}
where $\lfloor \theta_\alpha \rfloor$ and $\{ \theta_\alpha \}$ stand for the integer and fractional parts of $\theta_\alpha$. It follows from
the fact that the Euler Gamma function is strictly increasing on $[2,\infty[$ that, as soon as $K_n \geq 2$, $\Gamma\left(\theta_\alpha+K_n\right) > \Gamma\left(K_n\right)$. Consequently, as soon as $K_n \geq 2$, we deduce from \eqref{mgf_general} that for all $t>0$,
\begin{equation}
\label{PRMGFN0}
c_n(\alpha,\theta) \dE_{\alpha,0}[\exp(t K_n)] \leq m_n(t) \leq c_n(\alpha,\theta)  \big(\lfloor \theta_\alpha \rfloor + n \big)^{\theta_\alpha} \dE_{\alpha,0}[\exp(t K_n)].
\end{equation}
Furthermore, it is easy to see that
$$
\lim_{n\rightarrow \infty} \frac{1}{n}\log c_n(\alpha,\theta)=0.
$$
Hence, by taking the logarithm on both sides of \eqref{PRMGFN0}, we obtain from the squeeze theorem that for all $t >0$,
\begin{equation}
\label{CVGLAPLACEN0}
\lim_{n\rightarrow \infty} \frac{1}{n}\log m_n(t)= \lim_{n\rightarrow \infty} \frac{1}{n}\log \dE_{\alpha,0}[\exp(t K_n)]
=L_\alpha(t).
\end{equation}
On the other hand, assume that $-\alpha < \theta <0$. Since $-1 < \theta_{\alpha}<0$, $\theta_\alpha+K_n \geq 2$ as soon as $K_n \geq 3$, which implies that
$\Gamma\left(K_n\right)>\Gamma\left(\theta_\alpha+K_n\right)$. Moreover, we get from Gautschi's inequality that
$$
 \frac{\Gamma\left(\theta_\alpha+K_n\right)}{\Gamma(K_n)} > K_n^{\theta_\alpha} > n^{\theta_\alpha}.
$$
Therefore, as soon as $K_n \geq 3$, we obtain from \eqref{mgf_general} that for all $t>0$,
\begin{equation}
\label{PRMGFN00}
c_n(\alpha,\theta) n^{\theta_\alpha} \dE_{\alpha,0}[\exp(t K_n)] \leq m_n(t) \leq c_n(\alpha,\theta) \dE_{\alpha,0}[\exp(t K_n)].
\end{equation}
Finally, we deduce from \eqref{PRMGFN00}, together with the application of squeeze theorem, that the convergence \eqref{CVGLAPLACEN0} also holds in the case $-\alpha < \theta <0$, which completes the proof of Theorem \ref{T-LDP}.


\section{Proof of Theorem \ref{T-CI}}\label{sec3}
As for the proof of Theorem \ref{T-LDP}, we start by considering the special case $\theta=0$. In particular, for all $x \in (0,1)$ and for any $t>0$, it follows directly from Markov's inequality \cite{Ber(15)} that
\begin{equation}
\label{PRCI1}
 \dP(K_n \geq nx) =  \dP( \exp(t K_n) \geq \exp(n xt)) \leq \exp(-nxt) m_n(t).
\end{equation}
Hence, we immediately obtain from \eqref{ULBMNT} and \eqref{PRCI1} that for any $t>0$,
\begin{equation}
\label{PRCI2}
 \dP(K_n \geq nx) \leq \frac{1}{\alpha} \exp\big(\!\!-n(xt - L_\alpha(t) \big).
\end{equation}
Then, by taking the optimal unique value $t_x>0$ such that $L_\alpha^\prime(t_x)=x$, we find from \eqref{PRCI2} that
\begin{equation*}
 \dP(K_n \geq nx) \leq \frac{1}{\alpha} \exp\big(\!-nI_\alpha(x) \big)
\end{equation*}
as $I_\alpha(x)= xt_x -  L_\alpha(t_x)$. This completes the proof of Theorem \ref{T-CI} in the case $\theta=0$. Finally, the case $\theta >0$ and the case $-\alpha < \theta <0$ follow exactly along the same lines using \eqref{PRCI1} together with \eqref{PRMGFN0} and \eqref{PRMGFN00} respectively, which completes the proof of Theorem \ref{T-CI}.


\section{Concluding remarks}\label{sec4}

We proposed a new look on large deviations for the number $K_{n}$ of partition sets in the Ewens-Pitman model with $\alpha\in(0,1)$ and $\theta>-\alpha$, offering an alternative and concise proof of Feng-Hoppe large deviation principle. In addition, our approach allowed to establish a sharp concentration inequality for $K_{n}$ involving the rate function of the large deviation principle.

An interesting direction for future research is to refine Theorem \ref{T-LDP} in terms of sharp large deviations, in the spirit of the pioneering work of Bahadur-Rao on  the sample mean \cite{Bah(60)}; see also \cite{Bbr(24)} for sharp large deviations in the context of the number of descents in a random permutation, and the references therein. Along this direction, one may improve Theorem \ref{T-CI} by obtaining a more precise estimate of the pre-factor $P_{n}(\alpha,\theta)$. We argue that the moment-generating function \eqref{SHARPMNT1THETA_int} provides a promising starting point for developing both sharp large deviations and refined concentration inequalities for $K_{n}$. However, a lot of work remains to be done, particularly in the derivation of a sharp expansion of the three-parameter Mittag-Leffler function, which plays a crucial role in this analysis.

A large deviation principle for the number $K_{r,n}$ of partition subsets of size $r$ in the Ewens-Pitman model was first established in \cite{Fav(15)}, building on the original approach of \cite{Fen(98)}. A natural question is whether our approach can be adapted to provide an alternative proof of this result. If successful, this could lead to the derivation of a concentration inequality for $K_{r,n}$, which, to the best of our knowledge, remains an open problem in the literature. Furthermore, studying large deviations and concentration inequalities for functionals that involve both $K_{n}$ and $K_{r,n}$ is also a promising direction for future research.

\section*{Acknowledgement}

Stefano Favaro gratefully acknowledges the Italian Ministry of Education, University and Research (MIUR), “Dipartimenti di Eccellenza" grant 2023-2027.


\end{document}